\documentclass[12pt]{amsart}
\usepackage{amsmath}
\usepackage[all]{xy}
\usepackage{amssymb}
\usepackage{amscd}
\newtheorem{thm}{Theorem}[section]
\newtheorem{cor}[thm]{Corollary}
\newtheorem{lem}[thm]{Lemma}
\newtheorem{prop}[thm]{Proposition}
\theoremstyle{definition}
\newtheorem{defn}[thm]{Definition}
\newtheorem{ex}[thm]{Example}
\theoremstyle{remark}
\newtheorem{rem}[thm]{Remark}
\newtheorem{prob}[thm]{Problem}
\newtheorem{notation}[thm]{Notation}
\numberwithin{equation}{section}
\DeclareMathOperator{\N}{\mathbb{N}}
\DeclareMathOperator{\Z}{\mathbb{Z}}

\newcommand{\F}{\mathbb{F}}
\newcommand{\M}{\mathbb{M}}
\newcommand{\T}{\mathbb{T}}

\begin{document}

\author[A. C\^\i mpean]{Andrada C\^\i mpean}
\address{A. C\^\i mpean: "Babe\c s-Bolyai" University, Faculty of Mathematics and Computer Science, Str. Mihail Kog\u alniceanu 1, 400084, Cluj-Napoca, Romania}
\email{andrada.c@math.ubbcluj.ro; cimpean$\_$andrada@yahoo.com}

\author[P. Danchev]{Peter Danchev}
\address{P. Danchev: Institute of Mathematics and Informatics, Bulgarian Academy of Sciences \\ "Acad. G. Bonchev" str., bl. 8, 1113 Sofia, Bulgaria}
\email{danchev@math.bas.bg; pvdanchev@yahoo.com}

\title[$n$-Torsion Cleanness of Matrix Rings] {$n$-Torsion Clean and Almost $n$-Torsion Clean Matrix Rings}

\maketitle

\begin{abstract} We completely determine those natural numbers $n$ for which the full matrix ring $\mathbb{M}_n(\mathbb{F}_2)$ and the triangular matrix ring $\mathbb{T}_n(\mathbb{F}_2)$ over the two elements field $\F_2$ are either $n$-torsion clean or are almost $n$-torsion clean, respectively. These results somewhat address and settle a question, recently posed by Danchev-Matczuk in Contemp. Math. (2019) as well as they supply in a more precise aspect the nil-cleanness property of the full matrix $n\times n$ ring $\mathbb{M}_n(\F_2)$ for all naturals $n\geq 1$, established in Linear Algebra \& Appl. (2013) by Breaz-C\v{a}lug\v{a}reanu-Danchev-Micu and again in Linear Algebra \& Appl. (2018) by \v{S}ter as well as in Indag. Math. (2020) by Shitov.
\end{abstract}

\section{Introduction and Fundamental Tools}

Let $R$ be a ring possessing identity different to its zero element. As usual, for any positive integer $n$, the letters $\mathbb{M}_n(R)$ and $\mathbb{T}_n(R)$ will denote the full matrix ring and the (upper) triangular matrix ring, respectively.

For an arbitrary matrix $A$ over a {\it commutative} ring, we denote two polynomials associated to $A$ as follows: let $\chi_A(X)$ be the characteristic polynomial of $A$ defined standardly as $\chi_A(X)=\mathrm{det}(X\cdot I-A)$ where $X$ is the variable of the polynomial and $I$ represents the identity matrix -- thus $\chi_A(X)$ is a monic (i.e., its leading coefficient is $1$) polynomial of degree $n$, and let $\mu_A(X)$ be the minimal polynomial of $A$ defined as the monic polynomial of the smallest possible degree such that $\mu_A(A)=0$; so $\chi_A(X)$ is a multiple of $\mu_A(X)$. We shall hereafter designate for short $\chi_A(X)$ and $\mu_A(X)$ just as $\chi_A$ and $\mu_A$, respectively.

Moreover, for integers $a$ with $(a,n)=1$, let $l_a(n)$ denote the multiplicative order of $a(mod ~ n)$. If $(a,n)>1$, let $n_{(a)}$ denote the largest divisor of $n$ that is co-prime to $a$ and let we set $l^*_a:=l_a(n_{(a)})$. In particular, if $(a,n)=1$, then $l^*_a(n)=l_a(n)$.

As usual, for any prime integer $p$, the letter $\F_p=\mathbb{Z}_p$ will stand for the prime field of $p$ elements having characteristic $p$.

Letting $q$ be a monic polynomial over $\F_2$ with $q=X^n+c_{n-1}X^{n-1}+\ldots+c_1X+c_0$, we explicitly indicate the companion matrix associated to $q$ as the $n\times n$ matrix

$$
C=C_{c_0,c_1,\ldots, c_{n-1}}=\left(\begin{array}{ccccc}
0 & 0  &\ldots & 0 & -c_0 \\
1 & 0  &\ldots & 0 & -c_1\\
\vdots & \vdots  &\cdots  & \vdots & \vdots \\
0 & 0 &\ldots & 1 & -c_{n-1}
\end{array}\right).
$$

\medskip

\noindent To avoid some inaccuracies with the exact meaning, we also denote $C$ by $C_q$ using the subscript $q$ which may vary in each of the different cases.

\medskip

The following definition appeared in \cite{T}.

\begin{defn} Let $p$ be a prime. If the polynomial $X^n-1$ over the simple field of $p$ elements $\F_p$ has divisors of every degree less or equal to $n$, then $n$ is is said to be {\it $p$-practical}.
\end{defn}

The following useful technicality from number theory (see, e.g., \cite{T}), which will be used below without any further concrete referring, manifestly demonstrates more completely the importance of this notion, where $\phi(d)$ standardly denotes the Euler function of the integer $d$: {\it Suppose $p$ is a prime. An integer $n$ is $p$-practical if, and only if, every $m\in \N$ with $1\leq m \leq n$ can be written as $m=\sum_{d|n}l_p^*(d)n_d$, where $n_d$ is an integer with $0 \leq n_d \leq \frac{\phi(d)}{l_p^*(d)}$}.

\medskip

Concerning the classical theme of representing matrices as sums (and products of certain elements such as units, idempotents, nilpotents, etc.) one may indicate the following most important achievements like these: It was established in \cite{P1} and \cite{P2} that if $K$ is a field, then each element in $\mathbb{M}_n(K)$ is a linear combination of $3$ idempotents and, in particular, if $\mathrm{char}(K)$ is either $2$ or $3$, then every element of $\mathbb{M}_n(K)$ which is a sum of idempotents is actually a sum of four idempotents; in the case of fields $\F_2$ and $\F_3$, then any matrix over these two fields is a sum of three idempotents.

On the other vein, in \cite{BCDM} was concluded that an arbitrary matrix from $\mathbb{M}_n(\F_2)$ is a sum of a nilpotent and an idempotent. This fact was stated in proved in a more precise form in \cite{S} by establishing that the nilpotent must have an exponent not exceeding $4$ -- we will use the latter strengthening for our applicable purposes.

Moreover, some significant results in the subject, mainly attributed to Abyzov-Mukhametgaliev (see \cite{AM} and the bibliography herewith), were substantially improved in \cite{B} by proving that every matrix over a field of odd cardinality $k$ can be decomposed as a sum of a $k$-potent element and a nilpotent of order at most $3$.

And finally, in \cite{CD} were studied conditions on which presence there are certain decompositions of matrices over the fields $\F_2$ and $\F_3$.

\medskip

On the other side, mimicking \cite{DM}, for some arbitrary fixed $n\in \N$, a ring $R$ is said to be {\it $n$-torsion clean} if, for each $r\in R$, there exist a unit $u$ with $u^n=1$ and an idempotent $e$ such that $r=u+e$ and $n$ being the smallest possible positive integer having this (decomposable) property. Without the condition for minimality of $u$, this ring $R$ is just called {\it almost $n$-torsion clean}. For $n=2$ these two notions obviously do coincide. It was shown there that $\M_n(\F_2)$ is $m$-torsion clean for some natural number $m$ and also it had asked in which cases the equality $m=n$ is true.

\medskip

At first look, it is seemingly that the quoted above results are somewhat irrelevant each to other. Nevertheless, we shall demonstrate in the sequel that the presented facts are, however, closely related. So, the goal of the present paper is to determine exactly all naturals $n$ for which $\M_n(\F_2)$ and $\T_n(\F_2)$ are $n$-torsion clean and almost $n$-torsion clean, respectively, in terms of positive integers associated with the polynomial structure (especially, by concerning the divisibility of polynomials). Our achievements are the following: (1) For an arbitrary natural number $n\geq 1$, to show the existence of an integer $m$ from the segment $m\in [2,4]$ such that $\mathbb{M}_n(\F_2)$ is almost $m$-torsion clean. In particular, for some special naturals $n\in 4+8\mathbb{N}$, $\mathbb{M}_n(\F_2)$ is precisely $4$-torsion clean as well as $\mathbb{M}_2(\F_2)$ is always $2$-torsion clean. Even more generally, if $n$ is a $2$-practical integer, then $\mathbb{M}_n(\F_2)$ is almost $n$-torsion clean; (2) $\mathbb{T}_2(\F_2)$ is $2$-torsion clean as well as for an arbitrary $n\geq 3$, $\mathbb{T}_n(\F_2)$ is almost $n$-torsion clean $\iff$ $\mathbb{T}_n(\F_2)$ is $n$-torsion clean $\iff$ $n=2^l$ for $l\in \mathbb{N}$ with $l\not=0,1$ -- see our five major Theorems~\ref{alm}, \ref{mring}, \ref{TnAlmostNTorsionC}, \ref{thmTnMTorsionC} and \ref{thmTnNTorsionC} listed below.

Some similar questions concerned with fields of greater power will also be discussed in the sequel.

\section{Preliminary and Main Results}

Our further work is devoted to a comprehensive investigation of matrix presentations as sums of bounded torsion units (for some fixed exponent) and idempotents. Here we state our chief results, distributed into two subsections as follows:

\subsection{The full matrix ring}

First and foremost, we will completely settle the problem for $n$-torsion cleanness of matrix rings over the two element field. Specifically, the following is true:

\begin{thm}\label{alm} Let $m,n\in \mathbb{N}$. Then there exists $m\in \{2,3,4\}$ such that $\M_n(\F_2)$ is always almost $m$-torsion clean.

In addition, if $n\in 4+8\mathbb{N}$, then $\M_n(\F_2)$ is exactly $4$-torsion clean.
\end{thm}

\begin{proof} It was established in \cite{BCDM} and \cite{S} that the ring $\M_n(\F_2)$ is nil-clean for any $n\geq 1$ in the sense that any matrix $A$ with elements in $\F_2$ is presentable as a sum of an idempotent matrix $E$ and a nilpotent matrix $N$ of order of nilpotence at most $4$, say $N^4=0$. Therefore, as the characteristic in the matrix ring remains precisely $2$, we may represent that matrix $A$ as $A=(I_n+E)+(I_n+N)$, where $I_n$ is the identity matrix, and so the matrix $I_n+E$ remains an idempotent. But one readily sees that $(I_n+N)^4=I_n$, and so we arrive at the conclusion that $\M_n(\F_2)$ is almost $m$-torsion clean for some $m\in \N$ satisfying the two equalities $2\leq m\leq 4$, that is, $m$ lies in the set of three elements $\{2,3,4\}$, as stated.

The second part-half now follows immediately in view of the arguments stated above in combination with the main result in \cite{Sh}.
\end{proof}

Actually, the above assertion settles \cite[Question 1]{DM} in the negative, provided $n\geq 5$. Moreover, in addition, whether or {\it not} it can be deduced that $\M_n(\F_2)$ is $2$-torsion clean if, and only if, $n=2$ as well as that $\M_n(\F_2)$ is $3$-torsion clean if, and only if, $n=3$, are two still unsurmountable things at this stage.

\medskip

Next, treating the more complicated matrix structure of when $\M_n(\F_2)$ is almost $n$-torsion cleanness for an arbitrary natural number $n$, we begin here with the next statement which was established in \cite{BM}. Recall that $\N^*$ designates the union $\N\cup \{0\}$, where $\N$ is the set consisting of all naturals.

\begin{prop}\label{comp} (\cite{BM}) Let $n=m+k$ be a positive integer, where $m, k\in \mathbb{N}^*$. Fix constants
$c_0, c_1, \dots , c_{n_1} \in \F_2$ and denote $C=C_{c_0,c_1,...,c_{n-1}}$. For every polynomial
$g \in \F_2[X]$ of degree at most $n-2$ there exist two matrices $E,M \in \mathbb{M}_n(\F_2)$ such that
\begin{enumerate}
\item $C=E+M$;

\medskip

\item $E$ is a rank $k$ idempotent

\medskip

\noindent and

\medskip

\item $\chi_M = X^n + (k\cdot 1 + c_{n-1})X^{n-1} + g$.
\end{enumerate}
\end{prop}

The following technical claim is crucial for further applications.

\begin{lem}\label{lemLeqFive}
Suppose $p$ is an odd prime and $n=4p$. If $n$ is $2$-practical, then $l_2(p)\leq 5$. In particular, $l_2(7)=3$.
\end{lem}

\begin{proof}
Assume that $4p$ is $2$-practical with $p$ an odd prime. Therefore, every integer $m$ with $1\leq m \leq 4p$ can be written as $m=\sum_{d|n}l_2^*(d)n_d$, where $n_d$ is an integer with $0\leq n_d \leq \frac{\phi(d)}{l_2^*(d)}$. Henceforth, to demonstrate our initial assertion, we shall compute $l_2^*(d)n_d$ for every divisor $d$ of $n$. The divisors of $n$ are $1, 2, 4, p, 2p \mbox{ and }4p$.

\medskip

Let $d=1$. We have $l_2^*(1)=l_2(1_{(2)})=l_2(1)=1$, and since $\phi(1)=1$, it must be that $0\leq n_1\leq \frac{1}{1}$, so $n_1\in \{0,1\}$.

\medskip

Let $d=2$. We have $l_2^*(2)=l_2(2_{(2)})=l_2(1)=1$, and since $\phi(2)=1$, it must be that $0\leq n_2\leq \frac{1}{1}$, so $n_2\in \{0,1\}$.

\medskip

Let $d=4$. We have $l_2^*(4)=l_2(4_{(2)})=l_2(1)=1$, and since $\phi(4)=2$, it must be that $0\leq n_4\leq \frac{2}{1}$, so $n_4\in \{0,1,2\}$.

\medskip

Let $d=p$. We have $l_2^*(p)=l_2(p_{(2)})=l_2(p)$.

\medskip

Let $d=2p$. We have $l_2^*(2p)=l_2((2p)_{(2)})=l_2(p)$.

\medskip

Let $d=4p$. We have $l_2^*(4p)=l_2((4p)_{(2)})=l_2(p)$.

\medskip

Consequently, $\sum_{d|n}l_2^*(d)n_d=1\cdot n_1+1\cdot n_2+1\cdot n_4+l_2(p)n_p+l_2(p)n_{2p}+l_2(p)n_{4p}$. But since $n_1\in \{0,1\},$ $n_2\in \{0,1\}$, and $n_4\in \{0,1,2\}$, it follows at once that $n_1+n_2+n_4$ can be only in the set $\{0,1,2,3,4\}$. Therefore, the rest of $m=\sum_{d|n}l_2^*(d)n_d$, which is divided by $l_2(p)$, is in the set $\{0,1,2,3,4\}$. If $5<l_2(p)$, then for $m=5<n$ we have that the rest of $m$, divided by $l_2(p)$, would be exactly $5$, which is a contradiction. So, if $n=4p$ is $2$-practical, then $l_2(p)\leq 5$, as required. 
\end{proof}

The calculations given above unambiguously illustrate that $28$ is surely a $2$-practical number. In fact, over $\F_2$, the polynomial $x^{28}-1$ factors into a product of four degree $1$ polynomials and eight degree $3$ polynomials, and so has a factor of every degree (the direct check of this fact we leave to the interested reader for an inspection).

\medskip

The next notation will be used in what follows rather intensively.

\begin{notation}\label{k} Let $m>2$ be an integer. We shall denote by $k_1(m)$ the smallest number $k\in \{1,2,\dots,m\}$ such that the binomial $m\choose k$ is odd.
\end{notation}

The next lemma is well-known in the existing classical literature (see, e.g., \cite{E}) having an attractive and not too hard proof, so we will omit its details by leaving it to the readers for an eventual exercise.

\begin{lem}\label{NumberOfOdds} Let $m>2$ be an integer. Then the number of odd entries in the $m$-th line of the Pascal's Triangle is $2^v$, where $v$ is the number of digits $1$ in the binary representation of $m$.
\end{lem}

The following technicality somewhat describes the behavior of $k_1(m)$.

\begin{lem}\label{mEqualsk1} Let $m>2$ be an integer. Then the following two items hold:
 \begin{enumerate}
     \item $m=k_1(m)$ if, and only if, $m$ is a power of $2$.

     \medskip

     \item $m=k_1(m)$ or $\frac{m}{2}\geq k_1(m)$.
 \end{enumerate}
\end{lem}

\begin{proof}
\begin{enumerate}
\item Assuming $m=k_1(m)$, then the numbers of odd entries in the $m$-th line of the Pascal's Triangle is $2$, so by Lemma \ref{NumberOfOdds} we have that the number of digits $1$ in the binary representation of $m$ is $1$, so $m$ is a power of $2$.

    \medskip

    Assume now that $m$ is a power of $2$. Then the number of digits $1$ in the binary representation of $m$ is $1$, so we have exactly $2$ entries in the $m$-th line of the Pascal's Triangle. Therefore, $k_1(m)=m$.
\item Let $m>k_1(m)$. Assume $\frac{m}{2}<k_1(m)$. Then $0\neq k_2=m-k_1(m)<\frac{m}{2}<k_1(m)$ and ${m\choose k_2}={m\choose k_1(m)}$, so $m\choose k_2$ is odd, which manifestly contradicts the definition of $k_1(m)$.
    \end{enumerate}
\end{proof}

The next technicality is a purely number theoretic setting, which seems to the authors of the current paper to be absolutely "{\bf new}" and which could be of independent interest as well. Its proof could also be attacked via the classical instrument in number theory called {\it Lucas' theorem}.

\begin{lem}\label{k1PowerOf2}
Let $m>2$ be an integer. Then $k_1(m)=2^w$, where $w$ is a positive integer if $m$ is even and $w=0$ if $m$ is odd.
\end{lem}

\begin{proof} Write $m=2^w.t$ for some $w,t\in \mathbb{N}$ with $t$ odd. Firstly, if $w=0$, then $m=t$ is odd and we just can take $k_1(m)=1=2^w$, as wanted.

Secondly, given $m$ is even, we may assume that $k_1(m)\geq 2$. We differ two basic cases:

\medskip

\noindent{\bf Case 1:} Let $w=1$ and so $m=2t$. Since $\binom {m}{1}=m=2t$ is even and $\binom {m}{2}=\frac {2t(2t-1)}{2}=t(2t-1)$ is odd, we may choose $k_1(m)=2=2^w$, as needed.

\medskip

\noindent{\bf Case 2:} Let $m=2^w.t$ with $w\geq 2$. We shall now distinguish two major subcases as follows:

\medskip

\noindent{\bf Case 2.1:} Assume $k_1(m)$ is even, i.e., $k_1(m)=2j$ for some $j\in \mathbb{N}$. For all further conventions, in order to simplify the formulae, we shall use the notation $C[2^wt,2j]$ denoting by it the product of the multiples of the numerator of $\binom {2^wt}{2j}$ that are divisible by $2$, and divided by the product of the multiples divisible by $2$ of the denominator of $\binom {2^wt}{2j}$. One observes that

$$
\binom {m}{k_1(m)}=\binom {2^w.t}{2j}=\frac {C[2^w.t,2j].\lambda}{\mu},
$$

\medskip

\noindent where $\lambda$ and $\mu$ are odd positive integers. Moreover, one sees by a simple check that $C[2^w.t,2j]=\binom {2^{w-1}.t}{j}$, where $j<2^{w-1}$. Proceeding by induction on $w$, one may assume that our assertion holds for the coefficient $\binom {2^{w-1}.t}{j}$. In other words, the binomial $\binom {2^{w-1}.t}{2^{w-1}}$ is odd, while the binomial $\binom {2^{w-1}.t}{j}$ is even, provided $j<2^{w-1}$. But since $\binom {2^w.t}{2j}.\mu=\binom {2^{w-1}.t}{j}.\lambda$ and $\lambda,\mu$ are odd integers, we deduce that

$$
\binom {2^w.t}{2j}\equiv\binom {2^{w-1}.t}{j}~(mod~2).
$$

\medskip Therefore, the validity of our claim for $\binom {2^{w-1}.t}{j}$ will force the same for $\binom {2^w.t}{2j}$ as well, as desired.

\medskip

\noindent{\bf Case 2.2:} Assume $k_1(m)$ is odd such that $k_1(m)>1$, i.e., $k_1(m)=2j+1$, where $j\in \mathbb{N}$ and $j<2^{w-1}$. It is easily verified that

$$
\binom {m}{k_1(m)}=\binom {2^w.t}{2j+1}=\frac {\binom {2^w.t}{2j}.(2^w.t-2j)}{2j+1}.
$$

\medskip

\noindent As $j<2^{w-1}$, we have that $\binom {2^w.t}{2j}$ is even, so this equality implies that $\binom {2^w.t}{2j+1}$ is divisible by $4$, which is the pursued contradiction, so we are set.
\end{proof}

We come now to one of our main results describing when all matrices of sizes $n\geq 3$ over the two element field $\F_2$ are almost $n$-torsion clean. The result is closely related to \cite[Question 1]{DM}.

\begin{thm}\label{mring} Let $n>2$ be an integer. Then the following two items hold:
\begin{enumerate}
\item If $n$ is $2$-practical, then $\mathbb{M}_n(\F_2)$ is almost $n$-torsion clean.

\medskip

\item If $\mathbb{M}_n(\F_2)$ is almost $n$-torsion clean, then $n$ is not necessarily $2$-practical.
\end{enumerate}
\end{thm}

\begin{proof} We will differ and prove these two statements separately as follows:
\begin{enumerate}
\item Let $n>2$ be an integer that is $2$-practical. We will prove that any companion matrix of order $m$ with $1\leq m\leq n$ is almost $n$-torsion clean. Then, since any matrix is similar to its Frobenious normal form, a direct sum of companion matrices (i.e., a matrix with blocks, which is also diagonal), and since a direct sum of almost $n$-torsion clean companion matrices is almost $n$-torsion clean (because of the fact that a diagonal with almost $n$-torsion clean entries is always almost $n$-torsion clean and taking into account the fact that the diagonal consisting of blocks will keep the result), the conclusion will follow.

    In fact, let $m$ be an integer such that $1\leq m\leq n$. Since $n$ is $2$-practical, there exists a polynomial $r=X^m+r_{m-1}X^{m-1}+\dots+r_1X+r_0$  over $\F_2$ such that $r$ is a divisor of $X^n-1$.

    Let us now $C_q$ be an order $m$ companion matrix, $C_q=C_{c_0,c_1,\dots, c_{m-1}}$. We know that there exists $k\in \{1,2,\dots,n-1\}$ such that $k\cdot 1+c_{m-1}=r_{m-1}$. By virtue of Proposition~\ref{comp}, we know that for every $g\in F_2[X]$ of degree at most $m-2$ there exist a rank $k$ idempotent $E$ and a unit $U$ such that $C_q=E+U,$ with $\chi_U = X^m + (k\cdot 1 + c_{m-1})X^{m-1} + g$. It is not too hard to observe that we can choose $g$ such that $g=r_{m-2}X^{m-2}+\dots+r_1x+r_0$ then $\chi_U=r$. Since $\chi_U(U)=O_m$, it follows that $r(U)=O_m$. But $r$ is a divisor of $X^n-1$. Therefore, $U^n=I_m$ and so $C_q$ is almost $n$-torsion clean.

    \medskip

\item Suppose now to the contrary the implication "if $\mathbb{M}_n(\F_2)$ is almost $n$-torsion clean, then $n$ is $2$-practical" would be true. Next, in order to receive the desired contradiction, we will first of all prove that the implication "if $4\leq k_1(n)$, then $\mathbb{M}_n(\F_2)$ is almost $n$-torsion clean" is true -- actually, these two implications are independent each to other. Indeed, by what we have already shown so far, we know that $\mathbb{M}_n(\F_2)$ is a nil-clean ring with nil-clean index less than or equal to $4$ (see \cite{S} too). Letting $A\in M_n(F_2)$, then there exist an idempotent $E$ and a nilpotent $N$ such that $A=E+N$ with $N^4=O_n$. So, there will exist positive integers $k_1=k_1(n)<k_2<...<k_s=n$ such that $(I_n+N)^n=I_n+N^{k_1}+N^{k_2}+...+N^{k_s}$. Now, if we provide that $4\leq k_1(n)$, then $(I_n+N)^n=I_n$, and since $A=(I_n+E)+(I_n+N)$ is a clean decomposition of $A$, we will have actually gotten that $A$ is almost $n$-torsion clean, as required.

Furthermore, from validity of both implications "if $4\leq k_1(n)$, then $\mathbb{M}_n(\F_2)$ is almost $n$-torsion clean" and "if $\mathbb{M}_n(\F_2)$ is almost $n$-torsion clean, then $n$ is $2$-practical", we extract the following implication "{\it if $4\leq k_1(n)$, then $n$ is $2$-practical}".

What we intend to show now is that $4\leq k_1(44)$ and that $44$ is not $2$-practical. This contradiction will establish our desired claim as our former assumption will be wrong. To that goal, since $44=2^2\cdot 11$, we obtain that $4=k_1(44)$. Furthermore, there exists $p=11$ such that $44=4p$ holds. Assume now that $44$ is $2$-practical. Then, by Lemma \ref{lemLeqFive}, we have that $l_2(11)\leq 5.$ But the only odd prime divisors of $2^2-1, 2^3-1, 2^4-1, 2^5-1$ are $3, 7, 5, 31$. So, $l_2(11)\leq 5$ is an obvious contradiction, thus substantiating the wanted claim after all.
\end{enumerate}
\end{proof}

Some more comments comparing the present case of $n$-torsion cleanness with that of nil-cleanness could be of some interest and importance. In fact, it is not known if a nil-clean matrix over a field has all the companion matrices in its Frobenious normal form also nil-clean -- actually, it is a known fact only that if all companion matrices in the Frobenoious normal form of a matrix $A$ are nil-clean, then $A$ is nil-clean (for more details see \cite[Remark 9]{BM}). About the almost $n$-torsion clean case it is not known yet if an $n$-torsion clean matrix over a field (in particular, we are currently working over $\F_2$) has all the companion matrices in its Frobenious normal form also almost $n$-torsion clean. So, we are very interested if we can relate the almost $n$-torsion clean case with the problem of $n$-torsion cleanness of $\mathbb{M}_n(F_2)$ by asking of whether or not if $A\in \mathbb{M}_n(\F_2)$ is almost $n$-torsion clean, then it is not necessarily that any companion matrix in the Frobenious normal form of $A$ is also almost $n$-torsion clean.


\subsection{The triangular matrix ring}

Our next basic result, pertaining to the triangular matrix ring, asserts the following:

\begin{lem}\label{nilpIndexN} Let $n\in \N$. Then the nilpotency index of the nil-clean ring $\T_n(\F_2)$ is at most $n$.
\end{lem}

\begin{proof} Assume that any matrix $A\in \T_n(\F_2)$ is presentable as a sum of an idempotent matrix $E\in \T_n(\F_2)$ and a nilpotent matrix $N\in \T_n(\F_2)$ of order of nilpotence at most $k$, where $k$ is a number strictly less than $n$. Therefore, $N$ has only zeros on the main diagonal and on the first diagonal upper the main diagonal. Thus these two diagonals in $E$ are exactly the ones in $A$. Since $A$ can have any entries of $\F_2$, the upper left $2\times 2$ upper triangular matrix in $E$ can have any entries in $\F_2$. But this matrix has to be idempotent, because $E$ is an idempotent. Consequently, $\T_2(\F_2)$ is Boolean, which is demonstrably false.
\end{proof}

We recollect once again with accordance with Notation~\ref{k} above that $k_1(m)$ stands for the least integer $k\in \{1,2,\dots,m\}$ such that $m\choose k$ is odd.

The following technical claim is pivotal for our further development of results.

\begin{lem}\label{TnAlmostMTorsionC} Let $n>2$ and $m>2$ be two integers. Then $\T_n(\F_2)$ is almost $m$-torsion clean if, and only if, $n\leq k_1(m)$.
\end{lem}

\begin{proof} We first deal with the "left-to-right" implication. Given $A\in \T_n(\F_2)$, we know that $A$ is almost $m$-torsion clean, so that there exist an idempotent matrix $E\in \T_n(\F_2)$ and a unit matrix $U\in \T_n(\F_2)$ such that $A=E+U$ with $U^m=I_n$. Since $U$ is a unit of $\T_n(\F_2)$, then the entries in the main diagonal of $U$ are only ones and, therefore, there exists a nilpotent matrix $N\in \T_n(\F_2)$ such that $U=I_n+N$. So, $(I_n+N)^m=I_n$. Let us now $k_1<k_2<\dots<k_s=m$ be the integers in the set $\{1,2,\dots,m\}$ such that $m\choose k_i$ is odd for every $i\in \{1,2,\dots, s\}$. Therefore, $N^{k_1}+N^{k_2}+\dots+N^{k_s}=O_n$ and from here we derive that $N^{k_1}(I_n+N^{k_2-k_1}+\dots+N^{k_s-k_1})=O_n$. Since $N^{k_2-k_1},\dots,N^{k_s-k_1}$ are commuting nilpotents, their sum is again a nilpotent, and hence $I_n+N^{k_2-k_1}+\dots+N^{k_s-k_1}$ is a unit. Consequently,
$N^{k_1}=O_n$. With $A=(E+I_n)+N$ at hand, we infer that $N$ can be any nilpotent that appears in a nil-clean decomposition of $A$. Using now Lemma \ref{nilpIndexN} and $N^{k_1}=O_n$, it follows that $n\leq k_1$, as desired.

Conversely, assume that $A\in \T_n(\F_2)$. Knowing that $A$ is clean, there exist an idempotent matrix $E\in \T_n(\F_2)$ and a unit matrix $U\in \T_n(\F_2)$ such that $A=E+U$. Since $U$ is a unit of $\T_n(\F_2)$, then the entries in the main diagonal of $U$ are only ones and, therefore, there exists $N\in \T_n(\F_2)$ such that
$U=I_n+N$. We will now compute the power $(I_n+N)^m$. To that purpose, let $k_1<k_2<\dots<k_s=m$ be the integers in the set $\{1,2,\dots,m\}$ such that $m\choose k_i$ is odd for every $i\in \{1,2,\dots, s\}$. Furthermore, $(I_n+N)^m=I_n+N^{k_1}+N^{k_2}+\dots+N^{k_s}$, as wanted. But we know that $n\leq k_1$ and since $N^n=O_n$, it follows that $N^{k_1}=N^{k_2}=\dots=N^{k_s}=O_n$. So, $(I_n+N)^m=I_n$, whence $A$ is almost $m$-torsion clean.
\end{proof}

We now come to our first main result in this subsection.

\begin{thm}\label{TnAlmostNTorsionC} Let $n>2$ be an integer. Then $\T_n(\F_2)$ is almost $n$-torsion clean if, and only if, $n=2^l, l\in \mathbb{N}\setminus \{0,1\}$.
\end{thm}

\begin{proof} If we set $m=n$ in Lemma~\ref{TnAlmostMTorsionC}, we shall obtain that $\T_n(\F_2)$ is almost $n$-torsion clean if and only if $n\leq k_1(n)$. However, $n\leq k_1$ and $k_1\leq n$. So, $n=k_1(n)$. But such integers $n$ are, with the aid of Lemma~\ref{mEqualsk1}, only $2^l$ with $l\in \mathbb{N}\setminus \{0,1\}$. Finally, $\T_n(\F_2)$ is almost $n$-torsion clean if and only if $n=2^l, l\in \mathbb{N}\setminus \{0,1\}$, as claimed.
\end{proof}

The following lemma somewhat restate Lemma~\ref{TnAlmostMTorsionC} in a more convenient for us form of $m$-torsion clean rings like this:

\begin{lem}\label{LemTnMTorsionC} Let $n>2$ and $m>2$ be two integers. Then $\T_n(\F_2)$ is $m$-torsion clean if, and only if, the following two points hold:
\begin{enumerate}
\item $n\leq k_1(m)$.

\medskip

\item $n>k_1(u)$ for every integer $u\in \{2,3,\dots, m-1\}$.
\end{enumerate}
\end{lem}

\begin{proof} We just use Lemma~\ref{TnAlmostMTorsionC} accomplished with the definition of $m$-torsion cleanness.
\end{proof}

Arguing as above, we continue with a more precise description of $m$-torsion cleanness of triangular matrix rings.

\begin{lem}\label{mGEQn} Let $n>2$ and $m>2$ be two integers. If $\T_n(\F_2)$ is $m$-torsion clean, then $m\geq n$.
\end{lem}

\begin{proof} Assume in a way of contradiction $m<n$. Since $k_1(m)\leq m$, we have that $k_1(m)<n$, which contradicts $(1)$ from Lemma~\ref{LemTnMTorsionC}, as expected.
\end{proof}

The next lemma somewhat strengthens the previous one.

\begin{lem} Let $n>2$ and $m>2$ be two integers. If $\T_n(\F_2)$ is $m$-torsion clean, then $m$ is even.
\end{lem}

\begin{proof} We use the inequality $k_1(m)\geq n>2$ established above. So, $k_1(m)$ cannot be $1$, hence $n$ cannot be odd. Thus $m$ is necessarily even, as promised.
\end{proof}

We now have all the ingredients necessary to arrive at our other basic achievement of this subsection.

\begin{thm}\label{thmTnMTorsionC} Let $n>2$ and $m>2$ be two integers and let $t\geq 1$ be an other integer such that $2^t<n\leq 2^{t+1}$. Then $\T_n(\F_2)$ is $m$-torsion clean if, and only if, $m=2^{t+1}$.
\end{thm}

\begin{proof} Let $n>2$ and $m>2$ be two integers, $t\geq 1$ an other integer such that $2^t<n\leq 2^{t+1}$ and $\T_n(\F_2)$ is $m$-torsion clean. Applying Lemma~\ref{LemTnMTorsionC}, we obtain that
\begin{enumerate}
\item $n\leq k_1(m)$;

\medskip

\item $n>k_1(u)$ for every integer $u\in \{2,3,\dots, m-1\}$.

So, by the usage of points $(1)$ and $(2)$, it follows that $k_1(u)<k_1(m)$ for every integer $u\in \{2,3,\dots, m-1\}$.

Assuming now that $m$ is not a power of $2$, then there exists a positive integer $t'$ such that $2^{t'}<m<2^{t'+1}$. Since $2^{t'+1}>m$, it follows that $2^{t'}>\frac{m}{2}$. But $\frac{m}{2}\geq k_1(m)$, because $m$ is not a power of $2$ according to Lemma~\ref{mEqualsk1}. Hence, for $u=2^{t'}<m$, we have owing to Lemma~\ref{mEqualsk1} that $\frac{m}{2}\geq k_1(m)$, and so $k_1(u)=u>\frac{m}{2}\geq k_1(m)$, which is a contradiction. Therefore, $m=2^{t'},$ $t'\in \mathbb{N} \setminus \{0,1\}$. If, however, $t'\leq t$, then $m=2^{t'}\leq 2^t<n$ which contradicts Lemma~\ref{mGEQn}. Assume $t'>t+1$. We have $n\leq 2^{t+1}<2^{t'}=m$ and since $k_1(2^{t+1})=2^{t+1}$, one obtains that $n\leq k_1(2^{t+1})$. Hence for $u=2^{t+1}<m$ we derive that $k_1(u)\geq n$, contradicting point (2).

Conversely, letting $n>2$ and $m>2$ be two integers, $t\geq 1$ an other integer such that $2^t<n\leq 2^{t+1}$ and $m=2^{t+1}$, it follows by $m=2^{t+1}$ that $k_1(m)=m\geq n$, so $n\leq k_1(m)$. Let us now $u\in \{2,3,\dots, m-1\}$. Assume $k_1(u)>2^t$. By application of Lemma~\ref{k1PowerOf2}, there exists a positive integer $d_u$ such that $k_1(u)=2^{t+d_u}\geq 2^{t+1}=m$. So, $k_1(u)\geq m>u$, which is demonstrably false as by definition $k_1(u)\leq u$. Consequently, $k_1(u)\leq 2^t<n$, whence $k_1(u)<n$. In conclusion, by Lemma~\ref{LemTnMTorsionC}, we have that $\T_n(\F_2)$ is $m$-torsion clean, as required.
\end{enumerate}
\end{proof}

The next example concretes somewhat the computations given above (compare with the proof of Theorem~\ref{mring}).

\begin{ex}
\begin{enumerate}
\item $\T_3(\F_2)$ is $4$-torsion clean.

\medskip

\item $\T_3(\F_2)$ is not $28$-torsion clean.
\end{enumerate}
\end{ex}

\begin{proof} It follows at once by the usage of Theorem~\ref{thmTnMTorsionC} that $\T_3(\F_2)$ is $m$-torsion clean if and only if $m=4$, substantiating both claims.
\end{proof}

One finally gets that the following statement is valid:

\begin{thm}\label{thmTnNTorsionC} Let $n>2$ be an integer. Then $\T_n(\F_2)$ is $n$-torsion clean if, and only if, $n=2^l, l\in \mathbb{N}\setminus \{0,1\}$.
\end{thm}

\begin{proof} If we set $m=n$ in Theorem~\ref{thmTnMTorsionC}, we conclude that $\T_n(\F_2)$ is $n$-torsion clean if and only if $n=2^l, l\in \mathbb{N}\setminus \{0,1\}$.
\end{proof}

Comparing Theorems~\ref{TnAlmostNTorsionC} and \ref{thmTnNTorsionC}, one deduces the following rather curious consequence.

\begin{cor} Suppose $n\geq 3$ is an integer. Then $\T_n(\F_2)$ is $n$-torsion clean if, and only if, $\T_n(\F_2)$ is almost $n$-torsion clean.
\end{cor}

The next comments could be of some interest and importance:

\begin{rem}\label{remAnotherWay} Another proof of Theorem~\ref{thmTnNTorsionC} may be drawn as follows: If we put $m=n$ in Lemma~\ref{LemTnMTorsionC}, we shall obtain that $\T_n(\F_2)$ is $n$-torsion clean if and only if the following two statements are fulfilled:
\begin{enumerate}
\item $n\leq k_1(n)$.

\medskip

\item $n>k_1(u)$ for every integer $u\in \{2,3,\dots, n-1\}$.
\end{enumerate}
Notice that the second statement is always true since $n>u\geq k_1(u)$, while the first one is true if and only if $n=k_1(n)$ if and only if $n=2^l, l\in \mathbb{N}\setminus \{0,1\}$. So, in conclusion, $\T_n(\F_2)$ is $n$-torsion clean if and only if $n=2^l, l\in \mathbb{N}\setminus \{0,1\}$.
\end{rem}

The subsequent commentaries below are worthwhile in a way to describe more completely all $n$-torsion clean matrices.

\begin{rem} In regard to all the work done so far, let us notice that, if $\mathbb{M}_n(\F_2)$ is almost $n$-torsion clean, then there exists a matrix $A\in \mathbb{M}_n(\F_2)$ such that the rank of the idempotent in the almost $n$-torsion clean decomposition of $A$ lies in $\{1,2,\dots, n-1\}$. Assuming the contrary, the rank of the idempotent in any almost $n$-torsion clean decomposition of every $A\in \mathbb{M}_n(\F_2)$ has rank $0$ or $n$. Thus, $A^n=I_n$ or $(A+I_n)^n=1$. Then, $det(A)$ and $det(A+I_n)$ are, surely, both not zeros. Therefore, $0$ and $1$ are not eigenvalues of $A$, for every $A\in \mathbb{M}_n(\F_2)$, which is false. So, there exists an almost $n$-torsion clean decomposition having rank of the idempotent not belonging to $\{0,n\}$, and thus we will work resultantly with Proposition~\ref{comp}. Besides, it is unknown yet whether or {\it not} from Theorem~\ref{mring} can be deduced a criterion only in terms of the natural number $n$ for $\mathbb{M}_n(\F_2)$ to be almost $n$-torsion clean.

Concerning the case $n=2$, it was proved in \cite[Examples 2.3, 2.5]{DM} that $\T_2(\F_2)$ and $\M_2(\F_2)$ are both $2$-torsion clean.  Moreover, \cite[Example 2.7]{DM} accomplished with the discussion after Question 1 from there demonstrate that $\M_3(\F_2)$ is $3$-torsion clean as well as that $\M_4(\F_2)$ is $4$-torsion clean. These two facts are also immediate consequences of our Theorem~\ref{mring} alluded to above.
\end{rem}

We close with two left-open questions of some interest and importance. The first one is a common generalization to the already obtained results.

\begin{prob}\label{quest1} Determine those natural numbers $n$ for which $\mathbb{M}_n(\F_{2^n})$ is $n$-torsion clean, respectively almost $n$-torsion clean, by finding a necessary and sufficient condition.
\end{prob}

Now, mimicking \cite{D}, we will say that the element $r$ of a ring $R$ is {\it weakly $n$-torsion clean decomposed} if $r=u+e$ or $r=u-e$, where $u\in R$ with $u^n=1$ for some $n\geq 1$ and $e\in R$ with $e^2=e$. So, a ring $R$ is called {\it weakly $n$-torsion clean} if there is $n\in \N$ such that every element of $R$ has a weakly $n$-torsion clean decomposition and $n$ is the minimal possible natural in these two equalities. Without the limitation on minimality, $R$ is just called {\it almost weakly $n$-torsion clean}.

For instance, one deduces that $\mathbb{Z}_7=\mathbb{F}_7$ is both $6$-torsion clean and weakly $6$-torsion clean. In general, if $R$ is a ring with only trivial idempotents, the natural $n$ should be even in the case of $n$-torsion clean rings, since $0=(-1)+1$ is the unique presentation. In the weak case we may, however, have that $0=1-1$, so that things differ each to other.

Concerning $\mathbb{Z}_8$, it is both $2$-torsion clean and weakly $2$-torsion clean. However, $\mathbb{Z}_{10}$ being isomorphic to $\mathbb{Z}_2 \times \mathbb{Z}_5$ is $4$-torsion clean but weakly $2$-torsion clean.

\medskip

In that way, we end our work with the following challenging problem which is relevant to the discussion above by considering the more complicated version of (almost) weak $n$-torsion cleanness.

\begin{prob}\label{quest2} Let $R$ be a ring. Describe explicitly those naturals $n$ for which $\mathbb{M}_n(R)$, respectively $\mathbb{T}_n(R)$, is (almost) weakly $n$-torsion clean.
\end{prob}

The concrete examination of that query could begin by considering the three elements field $\F_3=\mathbb{Z}_3$.


\vskip3pc

\begin{thebibliography}{99}

\bibitem{AM}
A.N. Abyzov and I.I. Mukhametgaliev, {\it On some matrix analogs of the little Fermat theorem}, Mat. Zametki \textbf{101} (2017), 163--168.

\bibitem{B}
S. Breaz, {\it Matrices over finite fields as sums of periodic and nilpotent elements}, Lin. Alg. \& Appl. \textbf{555} (2018), 92--97.

\bibitem{BCDM}
S. Breaz, G. C\v{a}lug\v{a}reanu, P. Danchev and T. Micu, {\it Nil-clean matrix rings}, Lin. Alg. \& Appl. \textbf{439} (2013), 3115--3119.

\bibitem{BM}
S. Breaz, G.C. Modoi, {\it Nil clean companion matrices}, Lin. Alg. \& Appl. \textbf{489} (2016), 50--60.

\bibitem{CD}
A. C\^\i mpean and P. Danchev, {\it Weakly nil-clean index and uniquely weakly nil-clean rings}, Internat. Electron. J. Algebra \textbf{21} (2017), 180--197.

\bibitem{D}
P. V. Danchev, {\it Weakly $n$-torsion clean rings}, preprint, submitted.

\bibitem{DM}
P. Danchev and J. Matczuk, {\it $n$-Torsion clean rings}, Contemp. Math. \textbf{727} (2019), 71--82.

\bibitem{E}
A.W.F. Edwards, Pascal's Arithmetical Triangle, Oxford University Press, New York and London, 1987.

\bibitem{P1}
C. de Seguins Pazzis, {\it On decomposing any matrix as a linear combination of three idempotents}, Lin. Alg. \& Appl. \textbf{433} (2010), 843--855.

\bibitem{P2}
C. de Seguins Pazzis, {\it On sums of idempotent matrices over a field of positive characteristic}, Lin. Alg. \& Appl. \textbf{433} (2010), 856--866.

\bibitem{Sh}
Y. Shitov, {\it The ring $\mathbb{M}_{8k+4}(\Z_2)$ is nil-clean of index four}, Indag. Math. \textbf{30} (2019), 1077--1078.

\bibitem{S}
J. \v{S}ter, {\it On expressing matrices over $\Z_2$ as the sum of an idempotent and a nilpotent}, Lin. Alg. \& Appl. \textbf{544} (2018), 339--349.

\bibitem{T}
L. Thompson, {\it Variations on a question concerning the degrees of divisors of $x^n-1$}, J. Th. Nombr. Bordeax \textbf{26} (2014), 253--267.

\end{thebibliography}
\end{document}